\documentclass[11pt]{article}
\usepackage{amssymb, amsmath, amsthm, amsfonts,xcolor, enumerate}
\usepackage{tikz}
\usetikzlibrary{patterns}
\usetikzlibrary{shapes}
\usepackage{fullpage}
\usepackage{hyperref}
\newtheorem{theorem}{Theorem}[section]
\newtheorem{lemma}[theorem]{Lemma}

\newtheorem{prop}[theorem]{Proposition}
\theoremstyle{definition}

\newcommand{\ep}{\varepsilon}

\newcommand{\cH}{\mathcal{H}}

\newcommand{\cF}{\mathcal{F}}
\newcommand{\cG}{\mathcal{G}}

\newcommand{\cP}{\mathcal{P}}
\newcommand{\cI}{\mathcal{I}}
\newcommand{\cT}{\mathcal{T}}

\newcommand{\bP}{\ensuremath{\mathbb{P}}}
\newcommand{\bE}{\ensuremath{\mathbb{E}}}
\newcommand{\bF}{\ensuremath{\mathbb{F}}}

\newcommand{\card}[1]{\left| #1 \right|}

\title{Intersecting families of discrete structures are typically trivial}
\author{
J\'ozsef Balogh
\thanks{ Department of Mathematics, University of Illinois, Urbana, IL 61801, USA and Bolyai Institute, University of Szeged, Szeged, Hungary {\tt jobal@math.uiuc.edu}.
    Research is partially supported by a Simons Fellowship, NSF CAREER Grant DMS-0745185, and Marie Curie FP7-PEOPLE-2012-IIF 327763.}
\and
Shagnik Das
\thanks{Department of Mathematics, ETH, 8092 Zurich, Switzerland {\tt shagnik@ucla.edu}.}
\and
Michelle Delcourt
\thanks{Department of Mathematics,
University of Illinois, Urbana, Illinois 61801, USA {\tt delcour2@illinois.edu}. Research supported by NSF Graduate Research Fellowship DGE 1144245.}
\and
Hong Liu
\thanks{Department of Mathematics,
University of Illinois, Urbana, Illinois 61801, USA {\tt hliu36@illinois.edu}.
}
\and
Maryam Sharifzadeh
\thanks{Department of Mathematics,
University of Illinois, Urbana, Illinois 61801, USA {\tt sharifz2@illinois.edu}.
}}

\begin{document}
\maketitle

\begin{abstract}
The study of intersecting structures is central to extremal combinatorics.  A family of permutations $\cF \subset S_n$ is \emph{$t$-intersecting} if any two permutations in $\cF$ agree on some $t$ indices, and is \emph{trivial} if all permutations in $\cF$ agree on the same $t$ indices.  A $k$-uniform hypergraph is \emph{$t$-intersecting} if any two of its edges have $t$ vertices in common, and \emph{trivial} if all its edges share the same $t$ vertices.

The fundamental problem is to determine how large an intersecting family can be.  Ellis, Friedgut and Pilpel proved that for $n$ sufficiently large with respect to $t$, the largest $t$-intersecting families in $S_n$ are the trivial ones.  The classic Erd\H{o}s--Ko--Rado theorem shows that the largest $t$-intersecting $k$-uniform hypergraphs are also trivial when $n$ is large.  We determine the \emph{typical} structure of $t$-intersecting families, extending these results to show that almost all intersecting families are trivial.  We also obtain sparse analogues of these extremal results, showing that they hold in random settings.

Our proofs use the Bollob\'as set-pairs inequality to bound the number of maximal intersecting families, which can then be combined with known stability theorems.  We also obtain similar results for vector spaces.
\end{abstract}

\section{Introduction} \label{sec:intro}
The fundamental problem in extremal combinatorics asks how large a system can be under certain restrictions.  Once resolved, this can then be strengthened by enumerating such systems and describing their typical structure.  In the context of graph theory, this study was initiated by Erd\H{o}s, Kleitman and Rothschild~\cite{ekr76} in 1976, who proved that almost all triangle-free graphs are bipartite.  In extremal set theory, a landmark result was the determination of the number of antichains among subsets of an $n$-element set by Kleitman~\cite{k69} in 1969.  These results have since inspired a great deal of research over the years, with many classical theorems having been so extended.

Intersecting hypergraphs were first studied in the seminal 1961 paper of Erd\H{o}s, Ko and Rado~\cite{ekr61}.  Not only have versions of the Erd\H{o}s--Ko--Rado theorem been obtained in various other settings, including permutations and vector spaces, but a great deal of modern research is still devoted to proving further extensions.  In this paper, we study intersecting families of discrete systems in various settings, determining their typical structure as $n$, the size of the underlying ground set, tends to infinity.

We will now present our results and briefly review the extremal results regarding intersecting families in these different settings.  We discuss permutations in Section~\ref{subsec-sn}, hypergraphs in Section~\ref{subsec-hyp} and vector spaces in Section~\ref{subsec-vs}.

In what follows, we write $\log$ for logarithms to the base $2$, and $\ln$ for logarithms to the base $e$.

\subsection{Permutations}\label{subsec-sn}
Denote by $S_n$ the symmetric group on $[n]$. A family of permutations $\cF\subseteq S_n$ is \emph{$t$-intersecting} if any two permutations in $\cF$ agree on at least $t$ indices; that is, for any $\sigma,\pi\in\cF$, $\card{\sigma \cap \pi} = |\{i\in [n]:\sigma(i)=\pi(i)\}|\ge t$.  When $t=1$, we simply call such families \emph{intersecting}.  A natural example of a $t$-intersecting family $\cF\subseteq S_n$ is a \emph{trivial} $t$-intersecting family, where there is a fixed $t$-set $I\subseteq [n]$ and values $\{ x_i : i \in I \}$ such that for every $\sigma \in\cF$ and $i\in I$, $\sigma(i)=x_i$. Ellis, Friedgut and Pilpel~\cite{efp} proved that, for $n$ sufficiently large with respect to $t$, a $t$-intersecting family $\cF\subseteq S_n$ has size at most $(n-t)!$, with equality only if $\cF$ is trivial. Our first result determines the typical structure of $t$-intersecting families in $S_n$, showing that trivial families are not just extremal but also typical.

\begin{theorem}\label{thm-perm-all}
For any fixed $t\ge 1$ and $n$ sufficiently large, almost all $t$-intersecting families of permutations in $S_n$ are trivial, and the number of $t$-intersecting families is $\left( \binom{n}{t}^2  t!+ o(1) \right) 2^{(n-t)!}$.
\end{theorem}

Additionally, we prove two extensions in the sparse setting of Theorem~\ref{thm-perm-all}. In the first we consider $t$-intersecting families of permutations of size $m$. Note that each maximum trivial $t$-intersecting family has $\binom{(n-t)!}{m}$ subfamilies of size $m$.  The following result shows that, for $m$ not too small\footnote{The lower bound on $m$ here is what we require in our calculations.  It would be interesting to determine how small $m$ can be for this statement to hold.}, the number of non-trivial $t$-intersecting families of $m$ permutations is a lower-order term.

\begin{theorem}\label{thm-perm-fixed}
For any fixed $t\ge 1$, $n$ sufficiently large and $n 2^{2n - 2t + 2} \log n \le m \le (n-t)!$, almost  all $t$-intersecting families of $m$ permutations in $S_n$ are trivial.
\end{theorem}

Secondly we obtain the following sparse extension of the result of Ellis, Friedgut and Pilpel~\cite{efp}.  Let $(S_n)_p$ denote the $p$-random subset of $S_n$, where each permutation in $S_n$ is included independently with probability $p$.  Provided $p$ is not too small, we show that with high probability the largest $t$-intersecting family in $(S_n)_p$ is trivial.  Note that the Ellis--Friedgut--Pilpel theorem corresponds to the case $p=1$.

\begin{theorem} \label{thm-perm-random}
For fixed $t \ge 1$, $n$ sufficiently large and $p = p(n) \ge \frac{800 n 2^{2n - 2t} \log n}{(n-t)!}$, with high probability every largest $t$-intersecting family in $\left(S_n \right)_p$ is trivial.
\end{theorem}

\subsection{Hypergraphs}\label{subsec-hyp}
For $k\ge 2$ and $1\le t<k$, a $k$-uniform hypergraph $\cH$ on vertex set $[n]$ is \emph{$t$-intersecting} if every pair of edges shares at least $t$ vertices. A family is \emph{trivial} if every edge in $\cH$ contains a fixed set of $t$ vertices. The classic Erd\H{o}s--Ko--Rado theorem~\cite{ekr61} and the work of Frankl~\cite{f76} and Wilson~\cite{w84} show that when $n \ge (t+1)(k-t+1)$, the largest $t$-intersecting $k$-uniform hypergraphs have $\binom{n-t}{k-t}$ edges, a bound attained (not uniquely) by trivial $t$-intersecting families.  We show that just beyond this bound, the trivial $t$-intersecting hypergraphs are typical.

\begin{theorem}\label{thm-hyp-count}
Let $n$, $k = k(n) \ge 3$ and $t = t(n) \ge 1$ be integers such that $n \ge (t+1)(k-t+1) + \eta_{k,t}$, where
\[ \eta_{k,t} = \begin{cases}
k + 8 \ln k &\mbox{for } t = 1, \\
12 \ln k &\mbox{for } t = 2 \mbox{ and } k-t \ge 3, \\
1 &\mbox{for } t \ge 3 \mbox{ and } k-t \ge 3, \\
31 &\mbox{for } t \ge 2 \mbox{ and } k-t = 2, \\
18 k &\mbox{for } t \ge 2 \mbox{ and } k - t = 1.
\end{cases}\]
Almost all $t$-intersecting $k$-uniform hypergraphs on $[n]$ are trivial, and the number of $t$-intersectng $k$-uniform hypergraphs is $\left( \binom{n}{t} + o(1) \right) 2^{\binom{n-t}{k-t}}$.
\end{theorem}

Observe that $\eta_{k,t} = 1$, which we have for most values of $t$ and $k$, is the best possible result, as when $n = (t+1)(k-t+1)$ the largest non-trivial $t$-intersecting hypergraphs are as large as the trivial hypergraphs.  In fact, there are many more of them, and hence for this $n$ almost every $t$-intersecting hypergraph is non-trivial.

However, there is no doubt that the case $t = 1$ is the most natural and interesting to study.  Theorem~\ref{thm-hyp-count} gives the asymptotic number of intersecting hypergraphs when $n \ge 3k + 8 \ln k$.  On the other hand, it is known that the trivial hypergraphs are the largest when $n \ge 2k$, and uniquely so when $n \ge 2k + 1$.  The following theorem, which we prove using spectral methods and the theory of graph containers, provides a slightly weaker result that covers the entire range.

\begin{theorem} \label{thm-hyp-containers}
For $k \ge 3$ and $n \ge 2k+1$, let $I(n,k)$ denote the number of intersecting $k$-uniform hypergraphs on $[n]$.  Then
\[ \log I(n,k) = \left( 1 + o(1) \right) \binom{n-1}{k-1}. \]
\end{theorem}

Similarly to permutations, we are able to obtain a sparse version of the Erd\H{o}s--Ko--Rado theorem\footnote{We could prove an analogue of Theorem~\ref{thm-perm-fixed}, but have decided to omit the very similar result.}.  Let $\cH^{k}(n,p)$ denote the $p$-random $k$-uniform hypergraph on $[n]$, in which every edge in ${[n]\choose k}$ is included independently with probability $p$. Balogh, Bohman and Mubayi~\cite{bbm} initiated the study of intersecting hypergraphs in the sparse random setting. Among other results, they determined the size of the largest intersecting subhypergraph of $\cH^k(n,p)$ when $k< n^{1/2-\ep}$. Recently, Gauy, H\`{a}n and Oliveira~\cite{RSA} determined the asymptotic size of the largest intersecting family for all $k$ and almost all $p$.  Hamm and Kahn~\cite{kh} obtained an exact result for $k< (\frac{1}{2}-\ep)(n\log n)^{1/2}$, for any constant $\ep$, determining for which $p$ we have with high probability that every largest intersecting subhypergraph of $\cH^k(n,p)$ is trivial.  We prove, provided $p$ is not too small, that the same conclusion holds even for $k$ as large as $n/4$.\footnote{A similar statement can be proved for $t$-intersecting families; we leave the details to the readers.}  We remark that Hamm and Kahn~\cite{hk} also studied the case $n=2k+1$ and $p=1-c$ for some constant $c > 0$.

\begin{theorem} \label{thm-hyp-random}
For $3 \leq k \leq \frac{n}{4}$, if
\begin{equation} \label{eqn-prob-bound}
p \ge p_0(n,k) = \frac{9 n \log \left(\frac{ne}{k} \right) {2k \choose k}{n \choose k}}{{n-k \choose k}^2},
\end{equation}
then with high probability every largest intersecting subhypergraph of $\cH^k(n,p)$ is trivial.
\end{theorem}

Observe that the lower bound on $p$ in~\eqref{eqn-prob-bound} is at most $9 n \log \left( \frac{ne}{k} \right) \left( \frac{2kn}{(n-k)^2} \right)^k$, and is thus exponentially small with respect to $k \log \left( \frac{n}{k} \right)$.

\subsection{Vector spaces}\label{subsec-vs}
Let $V$ be an $n$-dimensional vector space over a finite field $\bF_q$.  The number of $k$-dimensional subspaces in $V$ is given by the Gaussian binomial coefficient
\[{n \brack k}_q := \prod_{i=0}^{k-1}\frac{q^{n-i}-1}{q^{k-i}-1}.\]

A family $\cF$ of $k$-dimensional subspaces of $V$ is \emph{intersecting} if dim$(F_1 \cap F_2) \geq 1$ for all pairs of subspaces $F_1, F_2 \in \cF$. Hsieh \cite{h75} proved an Erd\H os--Ko--Rado-type theorem for vector spaces, showing that for $n \geq 2k + 1$, any intersecting family of $k$-dimensional subspaces of $V$ has size at most ${n-1\brack k-1}_q$. Furthermore, the only constructions achieving the maximum size are \emph{trivial}, consisting of all $k$-dimensional subspaces through a given $1$-dimensional subspace. The results we obtain for permutations and hypergraphs can be extended to vector spaces as well, and we determine here the typical structure of intersecting families of subspaces\footnote{Similarly to the permutations, sparse extensions can be proved, and we leave the details to the readers.}.

\begin{theorem}\label{thm-vec}
If $k \ge 2$, and either $q = 2$ and $n \ge 2k + 2$ or $q \ge 3$ and $n \geq 2k+1$, almost all intersecting families of $k$-dimensional subspaces of $\bF_q^n$ are trivial, and there are $\left( {n \brack 1}_q + o(1) \right) 2^{{n-1 \brack k-1}_q}$ intersecting families.
\end{theorem}

\medskip

The rest of paper is organised as follows.  In Section~\ref{sec-meth}, we outline our general method of the proofs, using intersecting hypergraphs as an illustrative example.  The subsequent sections contain the details needed in each particular setting: Section~\ref{sec-perm} deals with permutations, Section~\ref{sec-hyp} with hypergraphs, with Theorem~\ref{thm-hyp-containers} proven in Section~\ref{sec-containers}, and Section~\ref{sec-vs} contains the proof for our vector space result.  In Section~\ref{sec-cr}, we present some concluding remarks and open questions.

Our notation is standard.  We denote by $[n]$ the first $n$ positive integers, and by $[a,b]$ the integers between $a$ and $b$.  Given a set $X$, a $k$-subset is a subset of $X$ of size $k$, $\cP(X)$ represents the set of all subsets of $X$, and $\binom{X}{k}$ the set of all $k$-subsets.  Recall that we use $\log$ for the binary logarithm and $\ln$ for the natural logarithm.

\section{The general method}\label{sec-meth}

Our proofs consist of two stages.  We first obtain strong bounds on the number of maximal intersecting families.  An intersecting family is said to be \emph{maximal} if it is intersecting and is not contained in a larger intersecting family.  Given these bounds, we then use known stability results to bound the number of non-trivial families.  In this section we outline the ideas behind these steps, using intersecting $k$-uniform hypergraphs as a running example.  For clarity, we omit any involved calculations in this section; they shall be carried out in greater generality in Section~\ref{sec-hyp}.

\subsection{Maximal intersecting families} \label{sec-meth-max}

Given a family of sets $\cF$, we denote by $\cI(\cF) = \left\{ G \in \binom{[n]}{k} : \forall F \in \cF, \; G \cap F \neq \emptyset \right\}$ the family of all sets intersecting every set in $\cF$.  Note that $\cF$ forms an intersecting family if and only if $\cF \subset \cI(\cF)$, while $\cF$ is maximal if and only if $\cF = \cI(\cF)$.  Given a maximal intersecting family, we call $\cG \subset \cF$ a \emph{generating set} if $\cF = \cI(\cG)$.

Let $\cF_0 = \{F_1, F_2, \hdots, F_s \} \subset \cF$ be a minimal generating set of $\cF$.  Observe that, by the minimality of $\cF_0$, we have $\cF \subsetneq \cI \left(\cF_0 \setminus \{F_i \} \right)$ for each $1 \le i \le s$.  Hence for each $i$ we can find some set $G_i \in \cI \left( \cF_0 \setminus \{F_i\} \right) \setminus \cF$.  Since $G_i \in \cI \left(\cF_0 \setminus \{F_i\} \right)$, we have $G_i \cap F_j \neq \emptyset$ for all $i \neq j$.  Moreover, since $G_i \notin \cF = \cI(\cF_0)$, we must further have $G_i \cap F_i = \emptyset$.  Given these conditions, we may now apply Frankl's skew version~\cite{f82} of the celebrated Bollob\'as set-pairs inequality~\cite{b65} to bound the size of $\cF_0$.

\begin{theorem}[Frankl] \label{thm-setpairs}
Let $A_1, \hdots, A_m$ be sets of size $a$ and $B_1, \hdots, B_m$ be sets of size $b$ such that $A_i \cap B_i = \emptyset$ and $A_i \cap B_j \neq \emptyset$ for every $1 \le i < j \le m$.  Then $m \le \binom{a + b}{a}$.
\end{theorem}

Given our $k$-sets $\{F_i\}$ and $\{G_i\}$, we construct a system of set-pairs $\{(A_i,B_i)\}_{i=1}^{2s}$.  For $1 \le i \le s$, let $A_i = F_i$ and $B_i = G_i$, and for $s+1 \le i \le 2s$, let $A_i = G_{i-s}$ and $B_i = F_{i-s}$. One can check that the set pairs $\{(A_i,B_i)\}$ satisfy the conditions of Theorem~\ref{thm-setpairs}, and hence we deduce that $2s \le \binom{2k}{k}$, and so $\card{\cF_0} = s \le \frac12 \binom{2k}{k}$.

The fact that every maximal intersecting family admits a small generating set allows us to bound the number of maximal intersecting families.

\begin{prop} \label{prop-maxint}
The number of maximal intersecting $k$-uniform hypergraphs over $[n]$ is at most 
\[ \sum_{i=0}^{ \frac12 \binom{2k}{k}} \binom{\binom{n}{k}}{i} \le \binom{n}{k}^{\frac12 \binom{2k}{k}}. \]
\end{prop}

\begin{proof}
Map each maximal intersecting hypergraph $\cF$ to a minimal generating set $\cF_0 \subset \cF$.  As $\cF = \cI(\cF_0)$, this map is injective.  We have shown above that $\card{\cF_0} \le \frac12 \binom{2k}{k}$, and hence the number of maximal intersecting hypergraphs is bounded by the number of sets of at most $\frac12 \binom{2k}{k}$ edges, which is the sum above.
\end{proof}

\subsection{Enumeration} \label{sec-meth-enum}

Since any subset of a trivial family is itself trivial, it follows that every non-trivial intersecting family must be a subset of a maximal non-trivial family.  Suppose we have a stability result that not only shows that the trivial intersecting families are the largest, but bounds the size of the largest non-trivial family.  We can then use this stability result in conjunction with our bound on the number of maximal families to bound the total number of non-trivial intersecting families.

The following lemma, phrased in general terms that will be applicable in all of our settings, gives sufficient conditions for the trivial families to be typical.

\begin{lemma} \label{lem-cond}
Let $N_0$ denote the size of the largest trivial intersecting family, and let $N_1$ denote the size of the largest non-trivial intersecting family.  Suppose further that there are at most $M$ maximal intersecting families.  Provided
\begin{equation} \label{cond-allsizes}
\log M + N_1 - N_0 \rightarrow - \infty,
\end{equation}
almost all intersecting families are trivial.  Moreover, if $m$ is such that 
\begin{equation} \label{cond-fixedsize}
\log M - m \log \left( \frac{N_0}{N_1} \right) \rightarrow - \infty,
\end{equation}
then almost all intersecting families of size $m$ are trivial.
\end{lemma}

\begin{proof}
Since a largest trivial intersecting family has size $N_0$, and all of its subfamilies are also trivial, there are at least $2^{N_0}$ trivial families.  On the other hand, every non-trivial intersecting family is a subset of a maximal non-trivial intersecting family.  Each maximal non-trivial family has size at most $N_1$, and thus at most $2^{N_1}$ subfamilies.  Since there are at most $M$ maximal families, the number of non-trivial families is at most $M 2^{N_1}$.  The proportion of non-trivial families is thus at most $M 2^{N_1} / 2^{N_0}$, which tends to $0$ by \eqref{cond-allsizes}.  Hence, given \eqref{cond-allsizes}, almost all intersecting families are trivial.

For the second claim, observe that the number of trivial subfamilies of size $m$ is at least $\binom{N_0}{m}$ by considering subfamilies of one fixed trivial family.  On the other hand, each non-trivial family has at most $\binom{N_1}{m}$ subfamilies of size $m$, and hence there are at most $M \binom{N_1}{m}$ non-trivial families of size $m$.  We can thus bound the proportion of intersecting families of size $m$ that are non-trivial by 
\[ M \binom{N_1}{m} / \binom{N_0}{m} \le M \left( \frac{N_1}{N_0} \right)^m, \]
which tends to $0$ by \eqref{cond-fixedsize}.  Hence almost all intersecting families of size $m$ are trivial as well.
\end{proof}

Within the context of intersecting hypergraphs, the Erd\H{o}s--Ko--Rado theorem~\cite{ekr61} states that for $n > 2k$, the largest intersecting $k$-uniform hypergraphs over $[n]$ are trivial, having size $\binom{n-1}{k-1}$.  A stability result was given by Hilton and Milner~\cite{hm67}, who showed that for the same range, the largest non-trivial intersecting hypergraphs have size $\binom{n-1}{k-1} - \binom{n-k-1}{k-1} + 1$.  These thus give the values of $N_0$ and $N_1$ respectively, while $M$ is given by Proposition~\ref{prop-maxint}.

Finally, once having determined that almost all intersecting families are trivial, we will still have to count the number of such families.  The following lemma shows when the union bound over all maximal trivial families gives an asymptotical correct result.

\begin{lemma} \label{lem-union}
Let $T$ denote the number of maximal trivial intersecting families, and suppose they all have the same size $N_0$.  Suppose further that two distinct maximal families can have at most $N_2$ members in common.  Provided
\begin{equation} \label{cond-union}
2 \log T + N_2 - N_0 \rightarrow - \infty,
\end{equation}
the number of trivial intersecting families is $\left( T + o(1) \right) 2^{N_0}$.
\end{lemma}

\begin{proof}
Suppose $\cF_1, \hdots, \cF_T$ are the maximal trivial intersecting families.  Every trivial family is a subset of some $\cF_i$, and hence the collection of trivial families is given by $\cup_{i=1}^T \cP(\cF_i)$.  The Bonferroni inequalities state that, for any sets $\cG_1, \hdots, \cG_m$,
\[ \sum_i \card{\cG_i} - \sum_{i < j} \card{\cG_i \cap \cG_j} \le \card{\cup_i \cG_i} \le \sum_i \card{\cG_i}. \]
Applying this with $\cG_i = \cP ( \cF_i )$ for $1 \le i \le m = T$, we have $\card{\cG_i} = \card{\cP(\cF_i)} = 2^{N_0}$ and $\card{\cG_i \cap \cG_j} = \card{\cP(\cF_i \cap \cF_j)} \le 2^{N_2}$.  This gives
\[ \sum_i \card{\cG_i} = T \cdot 2^{N_0} \; \textrm{ and } \; \sum_{i<j} \card{\cG_i \cap \cG_j} \le 2^{N_2} \binom{T}{2} < 2^{2 \log T + N_2 - N_0} \cdot 2^{N_0} = o \left( 2^{N_0} \right), \]
from which the result follows.
\end{proof}

This framework, coupled with the appropriate extremal and stability theorems, allows us to obtain our results, although minor modifications are required in the various settings.  In the following sections we describe the necessary changes and present the calculations needed to apply Lemmas~\ref{lem-cond} and~\ref{lem-union}.

\section{Intersecting families of permutations}\label{sec-perm}

In this section, we furnish the details required in the setting of permutations.  Following the framework introduced in Section~\ref{sec-meth}, we first bound the number of maximal $t$-intersecting families of permutations, and then deduce from this Theorems~\ref{thm-perm-all},~\ref{thm-perm-fixed} and~\ref{thm-perm-random}.

\begin{prop} \label{prop-perm-max}
For any $n \ge t \ge 1$, the number of maximal $t$-intersecting families in $S_n$ is at most
\[ \sum_{i = 0}^{\frac12 \binom{2n - 2t + 2}{n - t + 1}} \binom{n!}{i} < n^{n 2^{2n - 2t + 1}}. \]
\end{prop}

\begin{proof}
Following the proof of Proposition~\ref{prop-maxint}, for a maximal $t$-intersecting family $\cF$, we define $\cI(\cF) = \left\{ \pi \in S_n : \forall \sigma \in \cF, \; |\pi \cap \sigma| \ge t \right\}$.
Let $\cF_0 = \{ \sigma_1, \hdots, \sigma_s \} \subset \cF$ be a minimal generating set.  By minimality, for each $1 \le i \le s$ we have some $\tau_i \in S_n$ such that $\card{\sigma_j \cap \tau_i} < t$ if and only if $i=j$.

To a permutation $\pi$ we may assign the $n$-set of pairs $H_{\pi} = \{ (1, \pi(1)), \hdots, (n, \pi(n)) \}$.  Observe that for any two permutations $\pi$ and $\pi'$, $\card{H_{\pi} \cap H_{\pi'}} = \card{\pi \cap \pi'}$.  Hence, if we denote $F_i = H_{\sigma_i}$ and $G_i = H_{\tau_i}$, we have $\card{F_i \cap G_j} < t$ if and only if $i = j$.

We now require the $t$-intersecting version of the Bollob\'as set-pairs inequality, proven by F\"uredi~\cite{f84}.

\begin{theorem}[F\"uredi] \label{thm-tsetpairs}
Let $A_1, \hdots, A_m$ be sets of size $a$ and $B_1, \hdots, B_m$ be sets of size $b$ such that we have $\card{A_i \cap B_i} < t$ and $\card{A_i \cap B_j} \ge t$ for $1 \le i < j \le m$.  Then $m \le \binom{a + b - 2t + 2}{a - t + 1}$.
\end{theorem}

We apply this to the sets $\{(A_i, B_i)\}_{i=1}^{2s}$, where for $1 \le i \le s$ we take $A_i = F_i$ and $B_i = G_i$, and for $s+1 \le i \le 2s$ we set $A_i = G_{i - s}$ and $B_i = F_{i - s}$.  The conditions of Theorem~\ref{thm-tsetpairs} are clearly satisfied, and hence we deduce $s \le \frac12 \binom{2n - 2t + 2}{n - t +1}$.

Thus, to every maximal family $\cF$ we may assign a distinct generating set of at most $\frac12 \binom{2n - 2t + 2}{n -t + 1}$ permutations, giving the above sum as a bound on the number of maximal families.  The upper bound follows since $n! \le n^n$ and $\binom{2n - 2t + 2}{n - t + 1} \le 2^{2n - 2t + 2}$.
\end{proof}

Given this bound, we apply Lemmas~\ref{lem-cond} and~\ref{lem-union} to prove our enumerative results.  Proposition~\ref{prop-perm-max} shows that we may take $M = n^{n 2^{2n - 2t + 1}}$.  Each trivial family, on the other hand, has to fix the images of $t$ indices.  There are $\binom{n}{t}$ ways to choose the indices, $\binom{n}{t}$ ways to choose their images, and $t!$ ways to assign the images to the indices, and thus $T = \binom{n}{t}^2 t!$ maximal trivial families.

The required extremal result is due to Ellis, Friedgut and Pilpel~\cite{efp}, who showed that for $n$ sufficiently large with respect to $t$, the largest $t$-intersecting families in $S_n$ are the trivial ones.  In a trivial family, $t$ indices are fixed, while we are free to permute the remaining $n-t$ indices.  Hence we have $N_0 = (n-t)!$.  Moreover, note that there are at least $t+1$ fixed indices in the intersection of two trivial families, and so $N_2 = (n-t-1)!$.  Finally, a stability result was obtained by Ellis~\cite{e11}, showing that when $t$ is fixed and $n$ tends to infinity, the largest non-trivial $t$-intersecting family has size $N_1 = \left( 1 - 1/e + o(1) \right) (n-t)!$.  We now proceed to prove Theorems~\ref{thm-perm-all} and~\ref{thm-perm-fixed}.

\begin{proof}[Proof of Theorem~\ref{thm-perm-all}]
We first apply Lemma~\ref{lem-cond} to show that almost all intersecting families are trivial.  We have
\[ \log M + N_1 - N_0 = n 2^{2n - 2t + 1} \log n - \left( 1/e + o(1) \right) (n-t)! \rightarrow - \infty, \]
and so~\eqref{cond-allsizes} is satisfied.  This shows that the number of non-trivial $t$-intersecting families is $o \left(2^{(n-t)!} \right)$.

We use Lemma~\ref{lem-union} to count the number of trivial families.  We see that~\eqref{cond-union} holds, since
\[ 2 \log T + N_2 - N_0 = 2 \log \left( \binom{n}{t}^2 t! \right) + (n-t-1)!-(n-t)! \le 4t \log (nt) - (n-t-1)(n-t-1)! \rightarrow - \infty. \]
Hence the number of trivial families is $\left( \binom{n}{t}^2 t! + o(1) \right) 2^{(n-t)!}$.  As the non-trivial families constitute a lower-order term, this completes the proof. 
\end{proof}

\begin{proof}[Proof of Theorem~\ref{thm-perm-fixed}]
To prove that almost every $t$-intersecting family of $m$ permutations is trivial, we show that~\eqref{cond-fixedsize} is satisfied.  Indeed, for $m \ge n2^{2n-2t+2} \log n$,
\begin{align*}
\log M - m \log \left( \frac{N_0}{N_1} \right) &= n2^{2n-2t+1} \log n - m \log \left( \frac{(n-t)!}{\left( 1 - 1/e + o(1) \right) (n-t)!} \right) \\
	& \le n 2^{2n - 2t + 1} \log n - 0.6 m \rightarrow - \infty.  \qedhere
\end{align*}
\end{proof}

Finally, we seek to prove Theorem~\ref{thm-perm-random}, showing that when $p \ge \frac{800 n 2^{2n - 2t} \log n}{(n-t)!}$, with high probability the largest $t$-intersecting family in the $p$-random set of permutations $(S_n)_p$ is trivial.

Let $\cT \subset S_n$ be a fixed maximal trivial family, and let $\cF_1, \hdots, \cF_M$ be the maximal non-trivial families.  Then the largest trivial family in $(S_n)_p$ has size at least $\card{(\cT)_p}$, while the largest non-trivial family has size $\max_i \card{(\cF_i)_p}$.  In expectation, $\bE \left[ \card{(\cT)_p} \right] = p\card{\cT} > p \card{\cF_i} = \bE \left[ \card{ (\cF_i)_p } \right]$, and our bound on $M$ is strong enough for a union bound calculation to go through.  We require the following version of Hoeffding's Inequality that is derived from~\cite[Theorem 2.3]{mcd}.

\begin{theorem}[Hoeffding] \label{thm-hoeffding}
Let the random variables $X_1, X_2, \hdots, X_n$ be independent, with $0 \le X_k \le 1$ for each $k$.  Let $X = \sum_{k=1}^n X_k$, let $\mu = \bE [X]$.  Then, for any $\ep > 0$,
\[ \bP \left( X \ge (1 + \ep) \mu \right) \le \mathrm{exp} \left( - \frac12 \ep^2 \mu \right) \quad \textrm{ and } \quad \bP \left( X \le ( 1- \ep) \mu \right) \le \mathrm{exp} \left( - \frac12 \ep^2 \mu \right). \]
\end{theorem}

\begin{proof} [Proof of Theorem~\ref{thm-perm-random}]
Let $(\cT)_p = \cT \cap (S_n)_p$, let $(\cF_i)_p = \cF_i \cap (S_n)_p$, and set $\ep = 1/10$.  Let $E_0$ be the event that $\card{(\cT)_p} < (1 - \ep)p \card{\cT} = (1 - \ep) p N_0$, and let $E_i$ be the event that $\card{(\cF_i)_p} > (1 + \ep) p N_1$.  Since $N_0 = (n-t)!$ and $N_1 = \left(1 - 1/e + o(1) \right) (n-t)!$, we have $(1 + \ep) p N_1 < (1 - \ep) p N_0$.  If there is a non-trivial largest $t$-intersecting family in $(S_n)_p$, we must have $\max_i \card{(\cF_i)_p} \ge \card{(\cT)_p}$, and so at least one of the events $E_j$, $0 \le j \le M$, must hold.

Now $\card{(\cT)_p} \sim \mathrm{Bin}(N_0,p)$, and so applying Theorem~\ref{thm-hoeffding} with $\mu = p N_0$, we have $\bP( E_0 ) \le \mathrm{exp} \left( - \frac{pN_0}{200} \right)$.  Similarly, for $1 \le i \le M$, $\card{(\cF_i)_p} \sim \mathrm{Bin}(\card{\cF_i}, p)$, where $\card{\cF_i} \le N_1$.  Let $X \sim \mathrm{Bin}(N_1,p)$.  Applying Theorem~\ref{thm-hoeffding} to $X$ with $\mu = p N_1$, we have
\[ \bP( E_i ) = \bP( \card{(\cF_i)_p} \ge (1 + \ep) p N_1 ) \le \bP ( X \ge (1 + \ep) p N_1) \le \mathrm{exp} \left( - \frac{p N_1}{200} \right). \]

Hence, by the union bound,
\[ \bP \left( \cup_{i=0}^M E_i \right) = \mathrm{exp} \left( - \frac{p N_0}{200} \right) + M \mathrm{exp} \left( - \frac{p N_1}{200} \right) \le \left(n^{n 2^{2n - 2t + 1}}+1\right)\cdot \mathrm{exp} \left( - \frac{ p N_1}{200} \right) = o(1) \]
when $p \ge \frac{800 n 2^{2n - 2t} \log n}{ (n - t)!} \ge \frac{200}{N_1} n 2^{2n - 2t + 1} \log n$.  Thus, for such $p$, the largest $t$-intersecting families in $(S_n)_p$ are trivial with high probability.
\end{proof}

\section{Intersecting hypergraphs} \label{sec-hyp}

We now turn our attention to $t$-intersecting hypergraphs, and seek to prove Theorems~\ref{thm-hyp-count} and~\ref{thm-hyp-random}.  The proof of Theorem~\ref{thm-hyp-containers} uses a different method, and is given in Section~\ref{sec-containers}.

We begin with a bound on the number of maximal $t$-intersecting hypergraphs.

\begin{prop} \label{prop-hyp-max}
The number of maximal $t$-intersecting $k$-uniform hypergraphs on $[n]$ is at most
\[ \sum_{i = 1}^{\binom{2(k-t)+1}{k-t}} \binom{\binom{n}{k}}{i} \le \binom{n}{k}^{\binom{2(k-t)+1}{k-t}}. \]
\end{prop}

\begin{proof}
The proof of this proposition follows the proof of Proposition~\ref{prop-maxint}, except we must replace Theorem~\ref{thm-setpairs} with its $t$-intersecting version Theorem~\ref{thm-tsetpairs}.  This shows that every $t$-intersecting hypergraph admits a minimal generating set of at most $\frac12 \binom{2(k-t) + 2}{k-t + 1} = \binom{2(k-t)+1}{k-t}$ edges.  Thus the map from maximal hypergraphs to their minimal generating sets injects into sets of at most $\binom{2(k-t)+1}{k-t}$ edges, resulting in the upper bound above.
\end{proof}

We shall now use Lemma~\ref{lem-cond} to show that almost every $t$-intersecting hypergraph is trivial.  Proposition~\ref{prop-hyp-max} supplies us with the value of $M$ required.  The Erd\H{o}s--Ko--Rado theorem \cite{ekr61} states that for $n$ sufficiently large, the largest $t$-intersecting hypergraphs are the trivial ones, which have size $N_0 = \binom{n-t}{k-t}$.  Wilson~\cite{w84} later showed that $n \ge (t+1)(k-t+1)$ was the correct bound.

Stability results for the Erd\H{o}s--Ko--Rado theorem have a long history, beginning with the Hilton--Milner theorem~\cite{hm67}, which resolved the $t = 1$ case.  After much incremental progress, Ahlswede and Khachatrian~\cite{ak96} completely determined the largest non-trivial intersecting hypergraphs for all ranges of parameters.  In our range of interest, $n \ge (t+1)(k-t+1)$, there are two possible largest non-trivial hypergraphs:
\begin{align*}
\cH_1 &= \left\{ F : \card{ F \cap [t+2]} \ge t + 1 \right\}, \textrm{ and } \\
\cH_2 &= \left\{ F : [t] \subset F, F \cap [t+1,k+1] \neq \emptyset \right\} \cup \left\{ [k+1] \setminus \{i\}: 1 \le i \le t \right\}.
\end{align*}

\begin{theorem}[Ahlswede--Khachatrian] \label{thm-ak}
Suppose $n \ge (t+1)(k-t+1)$.  If $k \le 2t + 1$, then the largest non-trivial $t$-intersecting $k$-uniform hypergraph over $[n]$ has size $\card{\cH_1}$.  If $k \ge 2t + 2$, then the largest non-trivial hypergraph has size $\max \left\{ \card{\cH_1}, \card{\cH_2} \right\}$.
\end{theorem}

This theorem provides the value of $N_1$ needed for Lemma~\ref{lem-cond}.  Before we proceed, we evaluate $\card{\cH_1}$ and $\card{\cH_2}$, making use of Pascal's identity for binomial coefficients.
\begin{align}
\card{\cH_1} &= (t + 2) \binom{n-t-2}{k-t-1} + \binom{n-t-2}{k-t-2} = \binom{n-t}{k-t} - \left( 1 - \frac{(t+1)(k-t)}{n - t - 1} \right) \binom{n-t-1}{k-t}. \label{size-H1} \\
\card{\cH_2} &= \binom{n-t}{k-t} - \binom{n-k-1}{k-t} + t. \label{size-H2}
\end{align}

In light of Theorem~\ref{thm-ak}, we have $N_1 \le \max \left\{ \card{\cH_1}, \card{\cH_2} \right\}$, which we estimate by
\begin{align}
N_1 &\le \max \left\{ \card{\cH_1}, \card{\cH_2} \right\} \notag \\
	&= \binom{n-t}{k-t} - \min \left\{ \left( 1 - \frac{(t+1)(k-t)}{n-t-1} \right) \binom{n-t-1}{k-t} , \binom{n-k-1}{k-t} - t \right\} \notag \\
	&\le \binom{n-t}{k-t} - \left( 1 - \frac{(t+1)(k-t)}{n-t-1} \right) \binom{n-k-1}{k-t} + t \le \binom{n-t}{k-t} - \frac{1}{n} \binom{n-k-1}{k-t} + n \label{size-both},
\end{align}
where the last inequality holds for $n \ge (t+1)(k-t+1) + 1$.  We shall also use the following inequality for $a \ge b \ge r$:
\begin{equation} \label{binom-ineq}
\frac{\binom{a}{r}}{\binom{b}{r}} = \prod_{j=0}^{r-1} \frac{a-j}{b-j} \ge \left( \frac{a}{b} \right)^r.
\end{equation}

Finally, to count the number of trivial families, we use Lemma~\ref{lem-union}.  Since each trivial family fixes $t$ elements, there are $T = \binom{n}{t}$ maximal trivial families.  The intersection of any two such families must fix at least $t+1$ elements, and so can have size at most $N_2 = \binom{n-t-1}{k-t-1}$.  With these preliminaries in place, we now prove Theorem~\ref{thm-hyp-count}.

\begin{proof}[Proof of Theorem~\ref{thm-hyp-count}]
We will first prove that, for $n, k$ and $t$ as in the statement of the theorem, almost all $t$-intersecting hypergraphs are trivial.  To this end, we verify that~\eqref{cond-allsizes} of Lemma~\ref{lem-cond} holds.

We start with the case $t = 1$.  The Hilton--Milner theorem states that when $n > 2k$, the largest non-trivial intersecting hypergraph is $\cH_2$, and so $N_1 = \card{\cH_2} = \binom{n-1}{k-1} - \binom{n-k-1}{k-1} + 1$.  Recall that the trivial hypergraphs have size $N_0 = \binom{n-1}{k-1}$.  Finally, since $\binom{n}{k} \le 2^n$, Proposition~\ref{prop-hyp-max} shows that we may use $\log M \le \binom{2k-1}{k-1} n$.

Hence, using~\eqref{size-H2} and~\eqref{binom-ineq}, we have
\[ \log M + N_1 - N_0 \le \binom{2k-1}{k-1} n - \binom{n-k-1}{k-1} + 1 \le \left( n - \left( \frac{n-k-1}{2k-1} \right)^{k-1} \right) \binom{2k-1}{k-1} + 1. \]

For $t = 1$, we have $n \ge (t + 1)(k - t + 1) + \eta_{k,t} = 2k + \eta_{k,1} = 3k + 8 \ln k$.  We may bound
\[ \left( \frac{n-k-1}{2k-1} \right)^{k-1} = \left( \frac{n-k-1}{2k-1} \right)^2 \left( \frac{n-k-1}{2k-1} \right)^{k-3} \ge \frac{n^2 }{16 k^2} \left( 1 + \frac{8 \ln k}{2k-1} \right)^{k-3}. \]

Since $1 + x \ge \mathrm{exp}(6x/11)$ for $x \le 1$, when $k$ is large we have 
\[ k^{-2} \left(1 + \frac{8 \ln k }{2k - 1} \right)^{k-3} \ge k^{-2} \mathrm{exp} \left(  \frac{48 (k-3) \ln k}{22k} \right) \ge k^{-2} \mathrm{exp} (2 \ln k ) = 1. \]
Thus there is some constant $c > 0$ such that $k^{-2} \left( 1 + \frac{8 \ln k}{2k - 1} \right)^{k-3} \ge c$ for all $k$, and thus $\left( \frac{n-k-1}{2k-1} \right)^{k-1} = \Omega(n^2)$.  Hence it follows that $\left( n - \left( \frac{n - k - 1}{2k - 1} \right)^{k-1} \right) \binom{2k-1}{k-1} - 1 \rightarrow - \infty$.

We next handle the case $k - t = 1$.  In this setting, we have $k \le 2t + 1$, and hence by Theorem~\ref{thm-ak}, the largest non-trivial hypergraph has size $N_1 = \card{\cH_1}$.  Using $\binom{n}{k} \le \left( \frac{ne}{k} \right)^k$, we may use $\log M \le k \log \left( \frac{ne}{k} \right) \binom{2(k-t)+1}{k-t}$.  Using~\eqref{size-H1} gives
\begin{align*}
\log M + N_1 - N_0 &\le k \log \left( \frac{ne}{k} \right) \binom{2(k-t) + 1}{k-t} - \left( 1 - \frac{(t+1)(k-t)}{n - t - 1} \right) \binom{n-t-1}{k-t} \\
	&= 3k \log \left( \frac{ne}{k} \right) - (n - 2k) = 3k \log \left( \frac{ne}{k} \right) + 2k - n.
\end{align*}
This expression is increasing in $k$.  Since we are assuming $n \ge (t + 1)(k-t + 1) + \eta_{k,t} = 2k + \eta_{k,k-1} = 20k$, we substitute $k = n/20$ to obtain $\log M + N_1 - N_0 \le \left( 3 \log (20 e) - 18 \right) n / 20 \rightarrow - \infty$, since $3 \log(20e) < 18$.

Similar calculations show that when $t \ge 2$ and $k -t = 2$, $\eta_{k,k-2} = 31$ suffices.  In this setting, we still have $N_1 = \card{\cH_1}$.  Using $\log M \le n \binom{2(k-t)+1}{k-t}$ and $n \ge (t+1)(k-t+1) + \eta_{k,k-2} > 3k$,
\[ \log M + N_1 - N_0 \le \left( 10 - \eta_{k,k-2} / 3 \right) n \rightarrow - \infty. \]

We now consider the remaining cases, when $t \ge 2$ and $k-t \ge 3$.  In this range, the largest non-trivial hypergraph has size $N_1 = \max \{ \card{\cH_1}, \card{\cH_2} \}$.  Using $\binom{n}{k} \le 2^n$, we have $\log M \le n \binom{2(k-t)+1}{k-t} \le 2n \binom{2(k-t)}{k-t}$.  By~\eqref{size-both} and~\eqref{binom-ineq}, and observing that $n - k - 1 \ge t (k-t) + \eta_{k,t}$, we have 
\begin{align}
\log M + N_1 - N_0 &\le 2n \binom{2(k-t)}{k-t} - \frac{1}{n} \binom{n-k-1}{k-t} + n \notag \\
&\le \left( 3n - \frac{1}{n} \left( \frac{n-k-1}{2(k-t)} \right)^{k-t} \right) \binom{2(k-t)}{k-t} \notag \\
&\le \left( 3n - \frac{n^2}{64 (k-t)^3} \left( \frac{t(k-t) + \eta_{k,t}}{2(k-t)} \right)^{k-t-3} \right) \binom{2(k-t)}{k-t}. \label{final-bound}
\end{align}

If $t = 2$, then $\eta_{k,t} = 12 \ln k$, and $\frac{t(k-t) + \eta_{k,t}}{2(k-t)} = 1 + \frac{6 \ln k}{k-2}$.  Using $1 + x \ge \mathrm{exp}(6x/11)$ again, we find that for large $k$, 
\[ (k-2)^{-3} \left( 1 + \frac{6 \ln k}{k-2} \right)^{k-5} \ge (k-2)^{-3} \mathrm{exp} \left( \frac{36 (k-5) \ln k }{11k} \right) \ge k^{-3} \mathrm{exp} (3 \ln k) = 1. \]
It follows that there is some constant $c > 0$ such that $(k-2)^{-3} \left( \frac{ 2(k-2) + \eta_{k,2}}{2(k-2)} \right)^{k-5} \ge c$ for all $k$.

If instead $t \ge 3$, then $(k-t)^{-3} \left( \frac{t(k-t) + \eta_{k,t}}{2(k-t)} \right)^{k-t-3} > (k-t)^{-3} (\frac32)^{k-t-3} \rightarrow \infty$ as $k-t \rightarrow \infty$, and thus there is some $c > 0$ such that $(k-t)^{-3} \left( \frac{t(k-t) + \eta_{k,t}}{2(k-t)} \right)^{k-t-3} \ge c$ for all $k > t$.  Hence, in either case, $\frac{n^2}{64 (k-t)^3} \left( \frac{t(k-t) + \eta_{k,t}}{2(k-t)} \right)^{k-t-3} = \Omega(n^2)$, and so from~\eqref{final-bound} it follows that $\log M + N_1 - N_0 \rightarrow - \infty$.

Thus our choice of $\eta_{k,t}$ ensures that for all $k > t$ we have $\log M + N_1 - N_0 \rightarrow - \infty$, satisfying~\eqref{cond-allsizes} of Lemma~\ref{lem-cond}, thus showing that almost all $t$-intersecting $k$-uniform hypergraphs are trivial.  To complete the proof of Theorem~\ref{thm-hyp-count}, we need only count the number of trivial hypergraphs.  By Lemma~\ref{lem-union}, it suffices to verify~\eqref{cond-union}.  We have
\[ 2 \log T + N_2 - N_0 = 2 \log \binom{n}{t} + \binom{n-t-1}{k-t-1} - \binom{n-t}{k-t} \le 2t \log \left( \frac{ne}{t} \right) - \binom{n-t-1}{k-t} \rightarrow - \infty \]
for $k-t \ge 2$ or $k -t = 1$ and $n \ge 20t$.  It follows that the number of $t$-intersecting $k$-uniform hypergraphs on $[n]$ is $\left( \binom{n}{t} + o(1) \right) 2^{\binom{n-t}{k-t}}$, as claimed.
\end{proof}

We conclude this section with the proof of Theorem~\ref{thm-hyp-random}, showing that even in sparse random hypergraphs, when the edge probability is as given in~\eqref{eqn-prob-bound} the largest intersecting subhypergraphs are trivial.

\begin{proof} [Proof of Theorem~\ref{thm-hyp-random}]
The proof follows that of Theorem~\ref{thm-perm-random}.  Let $\cT$ denote a fixed maximal trivial hypergraph, and let $(\cT)_p = \cT \cap \cH^k(n,p)$ be those edges of $\cT$ selected in $\cH^k(n,p)$.  Let $\cF_1, \cF_2, \hdots, \cF_M$ be the maximal non-trivial hypergraphs, where by Proposition~\ref{prop-hyp-max} we have $M < \binom{n}{k}^{\binom{2k-1}{k-1}} < 2^{k\log \left( \frac{ne}{k} \right) \binom{2k-1}{k-1}}$, and let $(\cF_i)_p = \cF_i \cap \cH^k(n,p)$ denote the corresponding random subhypergraphs.

Observe that $\card{\cT} = N_0 = \binom{n-1}{k-1}$, while by the Hilton--Milner theorem~\cite{hm67}, $\card{\cF_i} \le N_1 = \binom{n-1}{k-1} - \binom{n-k-1}{k-1} + 1$.  Setting $\tau = p\binom{n-k-1}{k-1}/ 3$, define events $E_0 = \left\{ \card{(\cT)_p} \le p N_0 - \tau \right\}$ and $E_i = \left\{ \card{(\cF_i)_p} \ge p N_1 + \tau \right\}$ for $1 \le i \le M$.  By our choice of $\tau$, if none of the events $\{ E_i \}_{i=0}^M$ occur then $\card{(\cT)_p} > \max_i \left\{ \card{ ( \cF_i)_p } \right\}$, and so the largest intersecting subhypergraphs in $\cH^k(n,p)$ are trivial.

Applying Theorem~\ref{thm-hoeffding}, we find
\[ \bP( E_0 ) \le \mathrm{exp}\left( - \frac{\tau^2}{2 p N_0} \right) \quad \textrm{and} \quad \bP( E_i ) \le \mathrm{exp} \left( - \frac{\tau^2}{2 p N_1} \right) \le \mathrm{exp} \left( - \frac{\tau^2}{2 p N_0} \right). \]

Hence, by the union bound,
\[ \bP \left( \cup_{i=0}^M E_i \right) \le (M+1) \mathrm{exp} \left( - \frac{\tau^2}{2 p N_0} \right) \le \left( 2^{k\log \left( \frac{ne}{k} \right) \binom{2k-1}{k-1} } + 1 \right) \mathrm{exp} \left( - \frac{p \binom{n-k-1}{k-1}^2 }{18 \binom{n-1}{k-1}} \right) \rightarrow 0 \]
when $p \ge p_0(n,k) = \frac{9 n \log \left( \frac{ne}{k} \right) \binom{2k}{k} \binom{n}{k}}{\binom{n-k}{k}^2 } \ge \frac{18k \log \left( \frac{ne}{k} \right) \binom{2k-1}{k-1} \binom{n-1}{k-1} }{ \binom{n-k-1}{k-1}^2 }$, giving the bound in~\eqref{eqn-prob-bound}.
\end{proof}

As proven in~\cite{bbm}, when $k \gg \sqrt{n \log \log n}$ and $\frac{\log n}{\binom{n-1}{k}} \ll p \ll \frac{e^{k^2 / 2n}}{\binom{n}{k}}$, a simple first moment argument shows that the largest intersecting subhypergraph of $\cH^k(n,p)$ is non-trivial with high probability.  This holds for $p$ considerably smaller than in~\eqref{eqn-prob-bound}, and it would be very interesting to determine the threshold at which trivial hypergraphs become the largest intersecting subhypergraphs of $\cH^k(n,p)$.

\section{Hypergraphs of large uniformity} \label{sec-containers}

Although Theorem~\ref{thm-hyp-count} provides very sharp results, it is somewhat incomplete in the case $t=1$, as we require $n \ge 3k + 8 \ln k$ instead of the Erd\H{o}s--Ko--Rado threshold $n \ge 2k+1$.  In this section we prove Theorem~\ref{thm-hyp-containers}, which fills in the gap with a slightly weaker result, providing the asymptotics of the \emph{logarithm} of the number of intersecting hypergraphs.

We combine spectral methods with the theory of graph containers\footnote{For the general theory of containers, we refer the reader to the papers of Balogh, Morris and Samotij~\cite{bms} and Saxton and Thomason~\cite{st}.} to prove this theorem.  Such an approach has previously been used for other enumerative problems in combinatorics; see, for instance, the work of Sapozhenko~\cite{sap01} or Alon, Balogh, Morris and Samotij~\cite{abms14}.  To use these methods in our setting, we exploit the connection between intersecting hypergraphs and independent sets in Kneser graphs.

The Kneser graph $KG(n,k)$ is a graph with vertex set $\binom{[n]}{k}$ and an edge between vertices $F_1, F_2 \in \binom{[n]}{k}$ if and only if $F_1 \cap F_2 = \emptyset$.  This graph has $N = \binom{n}{k}$ vertices and is $D$-regular, where $D = \binom{n-k}{k}$.  Moreover, subsets of vertices of $KG(n,k)$ correspond to $k$-uniform hypergraphs on $[n]$, and independent sets correspond directly to intersecting hypergraphs.  Our problem thus reduces to counting the number of independent sets in $KG(n,k)$.

The following graph containers theorem, appearing the form below in~\cite{klrs13}, provides a method to bound the number of independent sets in a graph.

\begin{theorem}[Kohayakawa--Lee--R\"odl--Samotij]\label{thm-containers}
Let $G$ be a graph on $N$ vertices, let $R$ and $\ell$ be integers, and let $\beta > 0$ be a positive real.  Then, provided
\begin{equation} \label{eqn-Rbound}
e^{- \beta \ell}N \le R,
\end{equation}
and, for every subset $S \subset V(G)$ of at least $R$ vertices, we have 
\begin{equation} \label{eqn-density}
e(S) \ge \beta \binom{\card{S}}{2},
\end{equation}
there is a collection of sets $C_i \subset V(G)$, $1 \le i \le \binom{N}{\ell}$, such that $\card{C_i} \le R + \ell$ for every $i$ and, for every independent set $I \subset V(G)$, there is some $i$ satisfying $I \subset C_i$.
\end{theorem}

The supersaturation condition of~\eqref{eqn-density} in Theorem~\ref{thm-containers} requires large vertex subsets to induce subgraphs of positive density.  We use spectral methods to show that the Kneser graph satisfies this property; a similar approach was used by Gauy, H\`an and Oliveira in~\cite{RSA}.  The expander-mixing lemma, due to Alon and Chung~\cite{ac88}, relates the eigenvalues of a graph to its distribution of edges.

\begin{theorem}[Alon--Chung]\label{thm-em-lemma}
Let $G$ be a $D$-regular graph on $N$ vertices, and let $\lambda$ be its minimum eigenvalue.  Then for all $S \subseteq V(G),$
\[ e(G[S]) \geq \frac{D}{2N}|S|^2 + \frac{\lambda}{2N}|S|\left(N - |S|\right). \]
\end{theorem}

To employ this result, we require the spectrum of the Kneser graph, which was determined in a seminal paper of Lov\'asz~\cite{lov79}.  In particular, the minimum eigenvalue of the Kneser graph $KG(n,k)$ is $\lambda = -\binom{n-k-1}{k-1} = -\frac{k}{n-k}D$.  Combined with Theorem~\ref{thm-em-lemma}, this gives the following supersaturation bound.

\begin{prop} \label{prop-supersaturation}
Given $\ep>0$, any set $S$ of at least $\left(1 + \ep \right) \binom{n-1}{k-1}$ vertices in the Kneser graph $KG(n,k)$ induces at least $\left( 1 - \frac{1}{1 + \ep} \right) \frac{Dn}{N(n-k)} \binom{\card{S}}{2}$ edges.
\end{prop}

\begin{proof}
Given a vertex set $S$ with $\card{S} \ge \left(1 + \ep \right) \binom{n-1}{k-1} = \left(1 + \ep \right) \frac{kN}{n}$, we apply Theorem~\ref{thm-em-lemma} and the fact that $\lambda = - \frac{k}{n-k} D$ to find
\begin{align*}
e(G[S]) &\geq \frac{D}{2N}|S|^2 + \frac{\lambda}{2N}|S|\left(N - |S|\right) \ge \left( \frac{D - \lambda}{N} + \frac{\lambda}{\card{S}} \right) \binom{\card{S}}{2} \geq \left( 1 - \frac{1}{1 + \ep} \right) \frac{Dn}{N(n-k)} \binom{\card{S}}{2}.
\end{align*}
\end{proof}

Having established supersaturation, we may now apply Theorem~\ref{thm-containers} to find a small set of containers of independent sets in the Kneser graph, from which we shall derive Theorem~\ref{thm-hyp-containers}.

\begin{prop} \label{prop-kneser-containers}
For $\ep > 0$ and $2 \le k \le \frac{n-1}{2}$, let $R = \left( 1 + \ep \right) \binom{n-1}{k-1}$ and
\[ \ell = \frac{1 + \ep}{\ep} \cdot \frac{(n-k) \binom{n}{k}}{n\binom{n-k}{k}} \ln \left( \frac{n}{(1 + \ep) k} \right). \]
Then there exist $k$-uniform hypergraphs $\cF_i$ over $[n]$, $1 \le i \le \binom{\binom{n}{k}}{\ell}$, each of size at most $R + \ell$, such that every intersecting $k$-uniform hypergraph $\cF$ over $[n]$ is a subhypergraph of $\cF_i$ for some $i$.
\end{prop}

\begin{proof}
We apply Theorem~\ref{thm-containers} to the Kneser graph $KG(n,k)$.  By Proposition~\ref{prop-supersaturation}, condition~\eqref{eqn-density} is satisfied by taking 
\[ \beta = \left( 1 - \frac{1}{1+\ep} \right) \frac{Dn}{N(n-k)}, \]
where $D = \binom{n-k}{k}$ and $N = \binom{n}{k}$.  In order to satisfy \eqref{eqn-Rbound}, we take
\[ \ell = \frac{1}{\beta} \ln \left( \frac{N}{R} \right) = \frac{1}{\beta} \ln \left( \frac{n}{(1 + \ep) k} \right)=\frac{1 + \ep}{\ep} \cdot \frac{(n-k) \binom{n}{k}}{n\binom{n-k}{k}} \ln \left( \frac{n}{(1 + \ep) k} \right). \]
Applying Theorem~\ref{thm-containers}, the result follows by taking $\cF_i$ to be the hypergraph with edges $C_i \subset \binom{[n]}{k}$, since every intersecting hypergraph is an independent set of $KG(n,k)$.
\end{proof}

We now derive Theorem~\ref{thm-hyp-containers}.
 
\begin{proof}[Proof of Theorem~\ref{thm-hyp-containers}]
Since there is an intersecting hypergraph of size $\binom{n-1}{k-1}$, and each of its subhypergraphs is also intersecting, we have a lower bound $\log I(n,k) \ge \binom{n-1}{k-1}$.  We therefore need to show that $\log I(n,k) \le \left(1 + o(1) \right) \binom{n-1}{k-1}$.  Using Proposition~\ref{prop-kneser-containers}, we will show that for any small $\ep > 0$, $\log I(n,k) \le (1 + 2\ep) \binom{n-1}{k-1}$, provided $n \ge 2k+1$ is sufficiently large with respect to $\ep$.

We know that every intersecting hypergraph is contained in one of $\binom{N}{\ell}$ containers, each of size at most $R + \ell$, where $R$ and $\ell$ are as in the statement of the proposition.  By a simple union bound, the total number of intersecting hypergraphs is at most $\binom{N}{\ell} 2^{R + \ell}$.  Therefore, since $N = \binom{n}{k}$,
\[ \log I(n,k) \le R + \ell + \ell \log \left( \frac{ N e }{\ell} \right) = R + \ell \log \left( \frac{2e \binom{n}{k}}{\ell} \right). \]

Because $R = \left( 1 + \ep \right) \binom{n-1}{k-1}$, it is enough to show that $\ell \log \left( \frac{2e \binom{n}{k}}{\ell} \right) \le \ep \binom{n-1}{k-1}$.  We have
\[ \ell = \frac{(1 + \ep) (n-k) \ln \frac{n}{(1 + \ep)k}}{\ep k \binom{n-k}{k} } \binom{n-1}{k-1} \le \frac{2 \frac{n}{k} \ln \frac{n}{k} }{\ep \binom{n-k}{k}} \binom{n-1}{k-1}, \]
and, provided $\ep < \frac{1}{20}$,
\[ \log \left( \frac{2e \binom{n}{k}}{\ell} \right) = \log \left( \frac{2 \ep e n}{(1 + \ep) (n-k) \ln \frac{n}{(1+\ep)k}} \cdot \binom{n-k}{k} \right) \le \log \binom{n-k}{k}. \]

Hence it suffices to have $2 \frac{n}{k} \ln \frac{n}{k} \le \ep^2 \binom{n-k}{k} / \log \binom{n-k}{k}$.  If $n = 2k+1$, the left-hand side is constant, while the right-hand side is $\Omega( n / \log n )$.  If $n \ge 2k +2$, the left-hand side is $O(n \log n)$, while the right-hand side is $\Omega ( n^2 / \log n)$, and thus the inequality holds for large enough $n$.

Letting $\ep \rightarrow 0$, we have $\log I(n,k) \le \left( 1 + o(1) \right) \binom{n-1}{k-1}$, as desired.
 \end{proof}

We conclude this section by observing that the $n \ge 2k+1$ bound in Theorem~\ref{thm-hyp-containers} is best possible.  When $n = 2k$, the $k$-sets in $[n]$ come in $\frac12 \binom{n}{k} = \binom{n-1}{k-1}$ complementary pairs, and a hypergraph is intersecting if and only if it does not contain both edges from a single pair.  We thus have $I(n,k) = 3^{\binom{n-1}{k-1}}$ when $n = 2k$.  For $n < 2k$, the complete hypergraph $\binom{[n]}{k}$ is itself intersecting, and thus $I(n,k) = 2^{\binom{n}{k}}$.

\section{Vector spaces}\label{sec-vs}

In this section we prove Theorem~\ref{thm-vec}, showing that almost all intersecting families of subspaces of a finite vector space are trivial.  We begin, as always, with a bound on the number of maximal families.

\begin{prop} \label{prop-vec-max}
The number of maximal intersecting families of $k$-dimensional subspaces of $\bF_q^n$ is at most
\[ \sum_{i=0}^{\binom{2k - 1}{k-1}} \binom{ {n \brack k}_q }{i} \le {n \brack k}_q^{\binom{2k-1}{k-1}}. \]
\end{prop}

\begin{proof}
Once again, we follow the strategy of Proposition~\ref{prop-maxint}, seeking to show that every maximal intersecting family of subspaces contains a minimal generating set of at most $\frac12 \binom{2k}{k} = \binom{2k-1}{k-1}$ subspaces.  We must replace Theorem~\ref{thm-setpairs} with its vector space analogue, proven by Lov\'asz~\cite{lov77} and appearing in the form below in~\cite{bf}.

\begin{theorem}[Lov\'asz] \label{thm-vecsetpairs}
Let $U_1, \hdots, U_m$ be $a$-dimensional and $V_1, \hdots, V_m$ be $b$-dimensional subspaces of a vector space $W$ over a field $\bF$ such that $U_i \cap V_i = \{ 0 \}$ and $U_i \cap V_i \neq \{ 0 \}$ for $1 \le i < j \le m$.  Then $m \le \binom{a+b}{a}$.
\end{theorem}

This gives a map from maximal intersecting families of subspaces to sets of at most $\binom{2k-1}{k-1}$ subspaces, resulting in the above bound.
\end{proof}

This proposition gives a value for $M$ to be used when applying Lemma~\ref{lem-cond}.  As stated in Section~\ref{subsec-vs}, the corresponding extremal result was proven by Hsieh~\cite{h75}, who showed that when $n \ge 2k+1$, the largest intersecting families are trivial, with size $N_0 = {n-1 \brack k-1}_q$.  Each trivial family fixes a one-dimensial subspace, and hence there are $T = {n \brack 1}_q$ maximal trivial families.  The intersection of any two fixes a two-dimensional subspace, and thus has size $N_2 = {n-2 \brack k-2}_q$.

A stability result was obtained by Blokhuis, Brouwer, Chowdhury, Frankl, Mussche, Patk\'os and Sz\H{o}nyi~\cite{hm-vs}, who determined the size of the largest non-trivial intersecting family to be $N_1 = {n-1 \brack k-1}_q - q^{k(k-1)} {n-k-1 \brack k-1} + q^k$ when $q \ge 3$ and $n \ge 2k + 1$ or $q = 2$ and $n \ge 2k + 2$.  With these results in hand, we prove Theorem~\ref{thm-vec}.

\begin{proof}[Proof of Theorem~\ref{thm-vec}]
We shall first verify that~\eqref{cond-allsizes} of Lemma~\ref{lem-cond} holds, thus showing that almost all intersecting families are trivial.  Since either $q \ge 3$ and $n \ge 2k+1$ or $q=2$ and $n \ge 2k+2$, the extremal and stability results hold, and thus $M \le {n \brack k}_q^{\binom{2k-1}{k-1}}$, $N_0 = {n-1 \brack k-1}_q$ and $N_1 = {n-1 \brack k-1}_q - q^{k(k-1)} {n - k - 1 \brack k-1}_q  + q^k$.  This gives
\begin{equation} \label{eqn-vec}
\log M + N_1 - N_0 = \log \left( {n \brack k}_q \right) \binom{2k-1}{k-1} + q^k - q^{k(k-1)} {n-k-1 \brack k-1}_q.
\end{equation}

We bound the Gaussian binomial coefficients above and below by
\[ q^{(n-k)k} \le {n \brack k}_q = \prod_{i=0}^{k-1} \frac{q^{n-i}-1}{q^{k-i}-1} \le (2q^{n-k})^k \]
and use the fact that $\binom{2k-1}{k-1} < 4^k$ to show that the right-hand side of~\eqref{eqn-vec} is at most
\begin{equation} \label{eqn-vec2}
k(n-k) 4^k \log \left( 2q \right)+ q^k - q^{k(k-1)} \cdot q^{(n-2k)(k-1)} \le n^2 4^k \log \left( 2q \right)+ q^k - q^{(n-k)(k-1)}.
\end{equation}

If $k = 2$, then the right-hand side of~\eqref{eqn-vec2} is $16n^2 \log(2q) + q^2 - q^{n-2} \rightarrow - \infty$ as $n \rightarrow \infty$.  On the other hand, if $3 \le k < n/2$ then $(n-k)(k-1) \ge 2(n-3)$, so the right-hand side of~\eqref{eqn-vec2} is bounded above by $n^2 2^n \log \left( 2q \right) + q^{n/2} - q^{2(n-3)} \rightarrow -\infty$.  In either case, we have $\log M + N_1 - N_0 \rightarrow - \infty$, and, by Lemma~\ref{lem-cond}, almost all intersecting families are trivial.

Now we need only show that there are $\left( {n \brack 1}_q + o(1) \right) 2^{{n-1 \brack k-1}_q}$ trivial families, which will follow by verifying~\eqref{cond-union} and applying Lemma~\ref{lem-union}.  We have
\begin{align*}
2 \log T + N_2 - N_0 &= 2 \log \left( {n \brack 1}_q \right) + {n-2 \brack k-2}_q - {n-1 \brack k-1}_q \\
	&\le 2n \log q - \left( 1 - \frac{q^{k-1} - 1}{q^{n-1} - 1} \right) {n-1 \brack k-1}_q \le 2n \log q - \frac12 q^{(k-1)(n-k)} \rightarrow - \infty,
\end{align*}
as required.  This completes the proof.
\end{proof}

\section{Concluding remarks}\label{sec-cr}

In this paper we study the typical structure of intersecting families within various settings in discrete mathematics.  We estimate the number of maximal intersecting families, and from these bounds are able to derive asymptotics of the total number of intersecting families and deduce sparse versions of the extremal results.  Nevertheless, numerous open problems remain.

One of the motivating problems behind this project was the sparse analogue of the Erd\H{o}s--Ko--Rado theorem.  We show that for $3 \le k \le n/4$ and $p$ not too small, the largest intersecting subhypergraphs of the random hypergraph $\cH^k(n,p)$ are trivial with high probability.  This extends previous results, which held for $k = O \left( \sqrt{n \log n} \right)$.  However, there is a considerable gap between our lower bound on $p$ in Theorem~\ref{thm-hyp-random} and the upper bound for which it is known that the sparse Erd\H{o}s--Ko--Rado theorem is false.  What happens in this intermediate range of probabilities?  Different techniques will also be required to study the problem for larger $k$; in this direction, Hamm and Kahn~\cite{hk} recently established the sparse result for $n = 2k+1$ and $p = 1 - c$ for some $c > 0$.

There is also the question of obtaining the sharp asymptotics on the number of intersecting $k$-uniform hypergraphs, $I(n,k)$.  Theorem~\ref{thm-hyp-count} gives these asymptotics for $n \ge 3k + 8 \ln k$, showing that almost all intersecting hypergraphs are trivial.  For $n \ge 2k+1$, Theorem~\ref{thm-hyp-containers} provides a slightly weaker result, showing $\log I(n,k) \approx \binom{n-1}{k-1}$.  New methods will be required to obtain the asymptotics of $I(n,k)$ itself for the complete range, as our bounds on the maximal intersecting families are not strong enough to apply when $n \le 3k$.  It is worth noting that when $n = 2k+1$, the typical intersecting families are non-trivial, as the Hilton--Milner families outnumber the trivial ones.  However, we suspect that $n \ge 2k+2$ may already suffice for the trivial families to become typical.

The problem of enumerating maximal intersecting structures is interesting in its own right.  Here we provide reasonably sharp upper bounds through the use of the Bollob\'as set-pairs inequality and its variants.  We can also obtain lower bounds of a similar nature.  For instance, form $t$-intersecting $k$-uniform hypergraphs by, for each bipartition $[2(k-t)+2] = X_1 \cup X_2$, selecting one of $X_1 \cup [2(k-t) + 3, 2k - t + 1]$ or $X_2 \cup [2(k-t) + 3, 2k-t+1]$ to be an edge in the hypergraph.  This gives $2^{\binom{2(k-t)+1}{k-t}}$ hypergraphs, each of which can be extended to a distinct maximal $t$-intersecting hypergraph over $[n]$, while Proposition~\ref{prop-hyp-max} gives an upper bound of $2^{n \binom{2(k-t)+1}{k-t}}$.  Related constructions give similar lower bounds in the permutation and vector space settings.  We believe the lower bounds to be closer to the truth, and more refined arguments could bridge the gap.  Recent work of Nagy and Patk\'os~\cite{np15} bridges the gap between the bounds in the case of intersecting hypergraphs.

Finally, one can pursue these ideas for various other extremal problems in discrete mathematics.  For example, we say that a family of permutations of $[n]$ is \emph{$t$-set-intersecting} if for every pair of permutations $\sigma, \pi$ there is some $t$-set $X \subset [n]$ such that $\sigma(X) = \pi(X)$.  Ellis~\cite{e12} proved that for $n$ sufficiently large, the biggest $t$-set-intersecting families are trivial; namely, they send a fixed set of $t$ indices to a fixed set of $t$ images.  Can we show that these trivial families are also typical?

\bigskip

\noindent \textbf{Acknowledgement}  We would like to thank the University of Szeged for their kind hospitality, and the referees for their careful reading of this paper.

\end{document}